\documentclass[a4paper, 12pt]{article}
\usepackage{mathrsfs}
\usepackage{amsmath}
\usepackage{amsfonts}
\usepackage[T1]{fontenc}
\usepackage[latin9]{inputenc}
\usepackage{amsthm}
\usepackage{graphicx}
\usepackage{amsmath}

\usepackage{amsthm}
\usepackage{array}
\usepackage{cases}
\usepackage{enumerate}
\usepackage{clrscode}
\usepackage{listings}
\usepackage[ruled,vlined,english,linesnumbered]{algorithm2e}
\usepackage{url}
\usepackage{enumerate,enumitem,todonotes}
\makeatletter
\def\th@plain{%
  \itshape 
}
\makeatother

\makeatletter
\renewenvironment{proof}[1][\proofname]{\par
  \pushQED{\qed}%
  \normalfont \topsep6\p@\@plus6\p@\relax
  \trivlist
  \item[\hskip\labelsep
        \bfseries
    #1\@addpunct{.}]\ignorespaces
}{%
  \popQED\endtrivlist\@endpefalse
}
\makeatother

\numberwithin{equation}{section}

\newtheorem{thm}{Theorem}[section]
\newtheorem{cor}[thm]{Corollary}

\newtheorem{lem}[thm]{Lemma}

\newtheorem{defn}{Definition}

\newtheorem{pblm}[thm]{Problem}

\newtheorem{obs}{Observation}
\newtheorem{rem}{\bf{Remark}}

\numberwithin{equation}{section}

\usepackage[varg]{txfonts}
\usepackage{graphicx, epsfig, subfigure}
\usepackage{enumerate} 
\usepackage[square, numbers, sort&compress]{natbib} 
\ifx\pdfoutput\undefined
 \usepackage[dvipdfm,%
  pdfstartview=FitH, 
  bookmarks=true,%
  bookmarksnumbered=true, 
  bookmarksopen=true, 
  plainpages=false,%
  pdfpagelabels,%
  colorlinks=true, 
  linkcolor=blue, 
  citecolor=blue,%
  urlcolor=black,
  pdfborder=001]{hyperref}
  \AtBeginDvi{}  
\else
 \usepackage[pdftex,%
  pdfstartview=FitH, 
  bookmarks=true,%
  bookmarksnumbered=true, 
  bookmarksopen=true, 
  plainpages=false,%
  pdfpagelabels,%
  colorlinks=true, 
  linkcolor=blue, 
  citecolor=blue,%
  urlcolor=black,
  pdfborder=001]{hyperref}
\fi

\numberwithin{equation}{section}

\newcommand{\defi}{\mathrm{def}_i}

\begin{document}
\title{\LARGE Defective incidence coloring of graphs 
\thanks{Supported by the National Natural Science Foundation of China (No.\,11871055).}
}
\author{Huimin Bi~~~~~~ Xin Zhang\thanks{Corresponding author. Email: xzhang@xidian.edu.cn.}\\
{\small School of Mathematics and Statistics, Xidian University, Xi'an, 710071, China}}


\maketitle

\begin{abstract}\baselineskip 0.60cm

We define the $d$-defective incidence chromatic number of a graph, generalizing the notion of incidence chromatic number,
and determine it for some classes of graphs including trees, complete bipartite graphs, complete graphs, and outerplanar graphs.
Fast algorithms for constructing the optimal $d$-defective incidence colorings of those graphs are presented.

\vspace{3mm}\noindent \emph{Keywords: incidence coloring; defective coloring; Latin square; outerplanar graph; polynomial-time algorithm}.
\end{abstract}

\baselineskip 0.60cm

\section{Introduction}\label{sec:1}

The incidence coloring of graphs, introduced by Brualdi and Quinn Massey \cite{MR1246668} in 1993, has been attracting attention of many researchers (see an online survey of Sopena \cite{Sopena} for the recent progresses of the study of the incidence coloring).
This coloring has many applications in theoretical computer science and information science as it can model the multi-frequency assignment problem where each transceiver can be simultaneously in sending and receiving modes \cite{MR4285045}.

Formally, let $G$ be a graph with vertex set $V(G)$ and edge set $E(G)$. An \textit{incidence} of $G$ is a vertex-edge pair $(v,e)$ such that the vertex $v$ is incident with the edge $e$. For a vertex $u\in V(G)$ and its neighbor $v$ in $G$ (say $v\in N_G(u)$), the incidence $(u,uv)$ is a \textit{strong incidence} of $u$, and the incidence $(v,uv)$ is a \textit{weak incidence} of $u$. We use $I_u$ and $A_u$ to denote the set of strong incidences and weak incidences of $u$, respectively.

Brualdi and Quinn Massey \cite{MR1246668} defined two \textit{adjacent incidences} as $(u , e)$ and $(w , f)$ such that
$u=w$ or $uw\in \{e,f\}$ (it may happen that $e=f$). They also defined the following.
\begin{defn}\label{def:1}
A \textit{proper incidence $k$-coloring} of $G$ is a mapping $\varphi$ from the set $I(G)$ of all incidences of $G$ to the set $[k]:=\{1,2,\ldots,k\}$ of integers in such a way that two adjacent incidences receive different colors.
\end{defn}

We think about the proper incidence coloring in another view of point, by giving an equivalent definition as follows.
\begin{defn}\label{def:2}
A proper incidence $k$-coloring of $G$ is a mapping $\varphi:I(G) \longrightarrow [k]$ 
such that the following conditions hold for every $u\in V(G)$: \vspace{-2mm}
\begin{enumerate}[label=\textbf{$(\alph*\ref{def:2})$}]\setlength{\itemsep}{-3pt}
\item \label{a1} $\varphi(u,uv)\neq \varphi(u,uw)$ for any two distinct vertices $v,w\in N_G(u)$;
\item \label{b1} every color in $\varphi(I_u):=\{\varphi(u,uv)~|~v\in N_G(u)\}$ does not appear among $A_u$.\vspace{-2mm}
\end{enumerate}
\end{defn}

We similarly define $\varphi(A_u):=\{\varphi(v,uv)~|~v\in N_G(u)\}$ and $\varphi(I_u\cup A_u):=\varphi(I_u)\cup \varphi(A_u)$, which will be frequently used throughout this paper.

The \textit{incidence chromatic number} of $G$, denoted by $\chi_i(G)$, is the minimum integer $k$ such that $G$ has a proper incidence $k$-coloring.
From either Definition \ref{def:1} or Definition \ref{def:2}, one can easily see that $\chi_i(G)\geq \Delta(G)+1$ for every graph $G$.
The best known upper bound for $\chi_i(G)$ is $\Delta(G)+20\log \Delta(G)+84$, due to Guiduli \cite{MR1428581}, who disproved
a conjecture of Brualdi and Quinn Massey \cite{MR1246668} that 
$\chi_i(G)\leq \Delta(G)+2$ for every graph $G$. This upper bound is also asymptotically sharp \cite{ALGOR198911,MR1428581}.

There are many interesting variations of incidence coloring of graphs, including incidence list coloring \cite{MR3981224}, incidence game coloring \cite{MR2522464,MR2594489}, interval incidence coloring \cite{MR3163167},
fractional incidence coloring \cite{MR2976372}, and oriented incidence coloring \cite{MR3878287}.

Motivated by Definition \ref{def:2}, we introduce the defective incidence coloring of graphs in this paper.
\begin{defn}\label{def:3}
A \textit{$d$-defective incidence $k$-coloring} of $G$ is a mapping $\varphi:I(G) \longrightarrow [k]$ 
such that the following conditions hold for every $u\in V(G)$: \vspace{-2mm}
\begin{enumerate}[label=\textbf{$(\alph*\ref{def:3})$}]\setlength{\itemsep}{-3pt}
\item \label{a} $\varphi(u,uv)\neq \varphi(u,uw)$ for any two distinct vertices $v,w\in N_G(u)$;
\item \label{b} $\varphi(u,uv)\neq \varphi(v,uv)$ for every  $v\in N_G(u)$;
\item \label{c} every color in $\varphi(I_u)$ appears at most $d$ times among $A_u$.\vspace{-2mm}
\end{enumerate}
\end{defn}
The minimum number of colors used among all $d$-defective incidence colorings of a graph $G$, denoted by $\chi^d_i(G)$, is the \textit{$d$-defective incidence chromatic number}. 

Comparing Definition \ref{def:2} with Definition \ref{def:3}, one can easily see that 
the $0$-defective incidence coloring is coincide with the proper incidence coloring, and thus $\chi^0_i(G)=\chi_i(G)$. 
Moreover,
\begin{align}\label{rela}
    \chi^0_i(G)\geq \chi^1_i(G)\geq \chi^2_i(G)\geq \cdots \geq \Delta(G).
\end{align}
This motivates us to define the \textit{incidence defectivity} $\defi(G)$ of a graph $G$.
Formally, we let
\begin{align*}
    \defi(G)=\min\{k~|~\chi^k_i(G)=\Delta(G)\}.
\end{align*}

This paper is organized as follows. 

In Section \ref{sec:cycle}, we show that the $d$-defective incidence chromatic number of a tree or a complete bipartite graph is its maximum degree whenever $d\geq 1$. Moreover, we construct linear-time algorithm to compute a $d$-defective incidence $\Delta$-coloring for every such graph with maximum degree $\Delta$. In Section \ref{sec:complete}, we establish an interesting relationship between the $1$-defective incidence colorings of the complete graph $K_n$ and the $n\times n$ Latin squares, and then prove that $K_n$ has a $d$-defective incidence $(n-1)$-coloring for every integer $d\geq 1$ whenever $n\neq 2,4$.
Thereafter, based on our new results on Latin squares, a quadratic-time algorithm is designed for constructing a $d$-defective incidence $(n-1)$-coloring of $K_n$ with $n\neq 2,4$. In Section \ref{sec:outerplanar}, we move our attention to the class of outerplanar graphs, a special graph class with bounded treewidth. 
We show that every outerplanar graph with maximum degree $\Delta\geq 4$ admits a $1$-defective incidence $\Delta$-coloring, and this bound for $\Delta$ is sharp.
Furthermore, we prove that every outerplanar graph with maximum degree $\Delta$ admits a $d$-defective incidence $\Delta$-coloring whenever $d\geq 2$, unless the graph is isomorphic to disjoint copies of $K_1$ and $K_2$. All colorings mentioned above can be constructed in polynomial time according to the proofs in Section \ref{sec:outerplanar}.

\section{Trees and complete bipartite graphs} \label{sec:cycle}

It is easy to observe that every path and every cycle has $1$-defective incidence chromatic number exactly 2, as we can color its incidences alternately by two colors following a fixed direction. Hence we have the following. 

\begin{thm}\label{thm:path}
If $n\geq 3$ is an integer, then $\chi^d_i(P_n)=2$ and $\chi^d_i(C_n)=2$ for every integer $d\geq 1$. \hfill$\square$
\end{thm}

A \textit{rooted tree} is a connected acyclic graph with a special vertex that is called the \textit{root} of the tree and every edge directly or indirectly originates from the root. In a rooted tree, every vertex $v$, except the root, has exactly one
\textit{parent vertex} $u$, which is the first vertex traversed on the
path from $v$ to the root. The vertex $v$ is called a \textit{child} of $u$. We use ${\rm Par}(v)$ to denote the parent vertex of $v$ in a rooted tree. An \textit{ordered tree} is a rooted tree in which an ordering is specified for the children of each vertex \cite{Bender}.
The \textit{depth} of a vertex in an ordered tree is the length of the (unique) path from
the root to the vertex. Note that the root has depth 0.

\begin{thm}\label{thm:tree}
If $T$ is a tree , then $\chi^d_i(T)=\Delta(T)$ for every integer $d\geq 1$.
\end{thm}

\begin{proof}
It is sufficient to show that $\chi^1_i(T)=\Delta(T)$.
Assume that $T$ has  already been embdded as an ordered tree whose root $r$ has the maximum degree $\Delta$.
We construct an incidence $\Delta$-coloring $\varphi$ of $T$ as follows.

First of all, for each child $u_i$ of $r$, let $\varphi(r,ru_i)=i-1$ (mod $\Delta$) and $\varphi(u_i,ru_i)=i$ (mod $\Delta$), where 
$u_i$ denotes the $i$-th child of $r$,
 And then for each vertex $u$ at depth $\ell\geq 1$, let $\varphi(u,uw_i)=\varphi(u,uv)+i$ (mod $\Delta$) and 
 $\varphi(w_i,uw_i)=\varphi(v,uv)+i$ (mod $\Delta$), where $w_i$ is the $i$-th child of $u$  and 
 $v$ is the parent of $u$. 
 
 We show that this incidence coloring of $T$ is $1$-defective.
Clearly, \ref{a} and \ref{b} by the construction of the coloring, and one can see that every two incidences of $A_u$ for any $u\in V(T)$ are colored distinctly. Hence \ref{c} with $d=1$ holds naturally. 
\end{proof}

Below we release a linear-time algorithm to construct a $1$-defective incidence $\Delta(T)$-coloring of $T$
based on the proof of Theorem \ref{thm:tree}.
Figure \ref{fig:Tree} shows an instance of Algorithm \ref{algo:tree} on how to construct a 1-defective incidence 3-coloring of a given tree with maximum degree three.

\begin{figure}[htp]
    \centering
    \includegraphics[width=12cm,height=8cm]{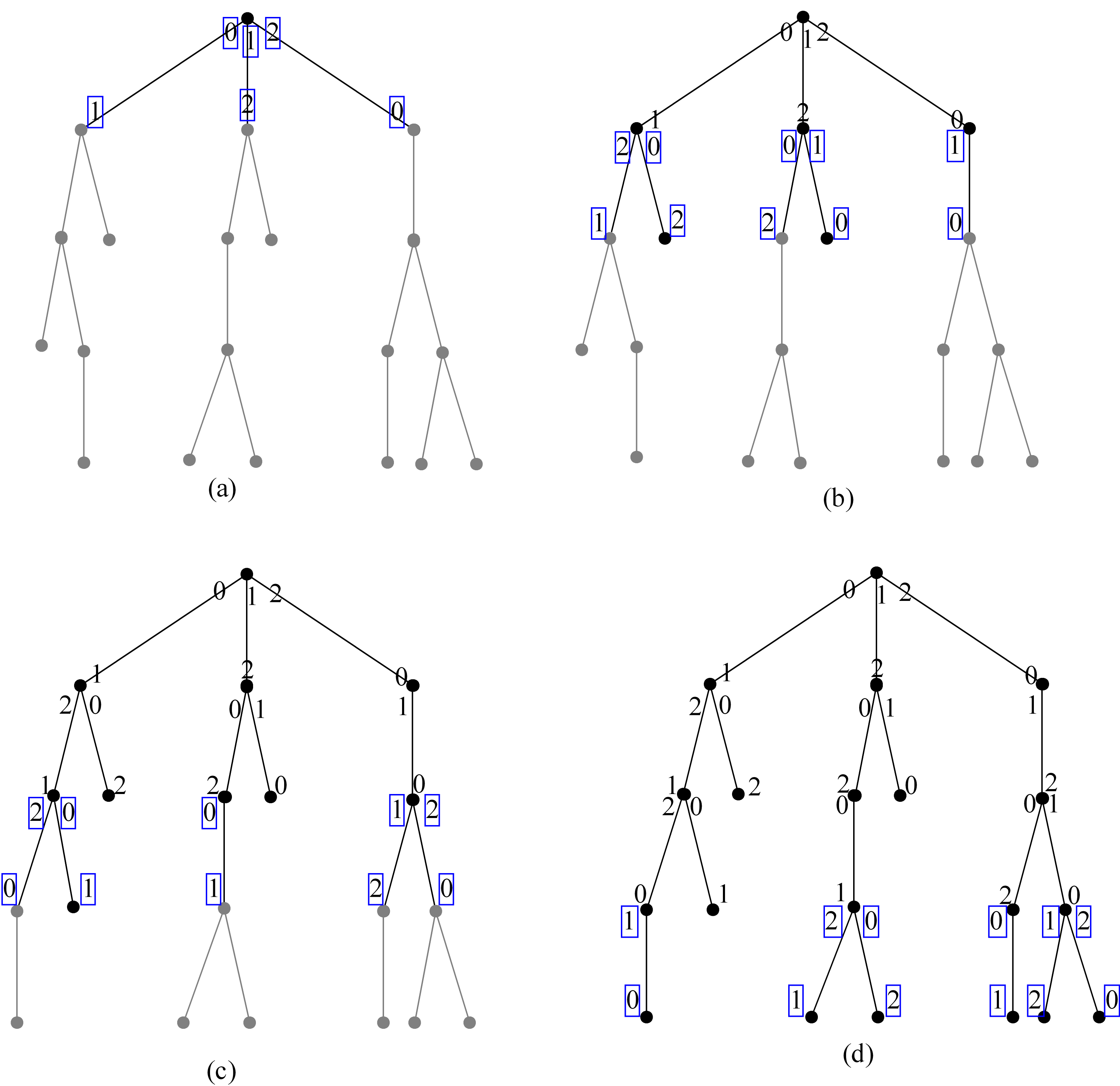}
    \caption{Constructing a 1-defective incidence 3-coloring using breadth-first search}
    \label{fig:Tree}
\end{figure}  

\begin{algorithm}[htp]
\BlankLine
\KwIn{A tree $T$ with maximum degree $\Delta$.}
\KwOut{A $1$-defective incidence coloring of $T$ using $\Delta$ colors. }
\BlankLine
Find a vertex $r$ of maximum degree in $T$\\
Root $T$ at $r$ using breadth-first search so that $r$ has a labelling $f(r)=(0)$ and each vertex $u$ at depth $\ell\geq 1$ has a labelling with an $(\ell+1)$-array $(x_0,\ldots,x_{\ell-1},x_{\ell})$ such that $(x_0,\ldots,x_{\ell-1})$ is the labelling of the parent of $u$ and $u$ is the $(x_\ell+1)$-th child of its parent\\
\tcc{$p(u)$ denotes below the last entry of the labelling of $u$.}
\For{$\ell=0$ to ${\rm Dep}(T)$}
{
\eIf{$\ell=0$}
{\For{each child $u$ of $r$}
{
$\varphi(r,ru) \gets\ p(u)$ (mod $\Delta)$\\
$\varphi(u,ru) \gets p(u)+1$  (mod $\Delta)$}
}{

\For{each vertex $u$ at depth $\ell$}
{
$v\gets {\rm Par}(u)$\\

\For {each child $w$ of $u$}
{
{$\varphi(u,uw) \gets\ \varphi(u,uv)+p(u)+1$ (mod $\Delta)$\\
$\varphi(w,uw) \gets\ \varphi(v,uv)+p(u)+1$ (mod $\Delta)$}
}
}
}
}
\caption{$1$-defective incidence coloring of $T$}\label{algo:tree}
\end{algorithm}


\begin{thm}\label{kmn}
$\chi^d(K_{m,n})= \max\{m,n\}=\Delta(K_{m,n})$ for every integer $d\geq 1$.
\end{thm}

\begin{proof}
Let $\{u_1,u_2,\ldots u_m\}$ and $\{v_1,v_2,\ldots v_n\}$ be 
the bipartition of $K_{m,n}$. Assume, without loss of generality, that $m\geq n$.
We construct an incidence $m$-coloring $m$ of $K_{m,n}$
by coloring $(v_i, u_jv_i)$ with $i+j-1$~(mod $m$), and
$(u_j, u_jv_i)$ with $i+j$~(mod $m$) for each $i\in [n]$ and $j\in [m]$.
Now we prove that this incidence coloring is 1-defective.

Clearly, \ref{a} and \ref{b} hold by the construction of $\varphi$. 

For each vertex $v_i$ with $i\in [n]$ and for each color $\ell\in \varphi(I_{v_i})$, if $\ell\in \varphi(A_{v_i})$, then by the construction of $\varphi$, there is a vertex $u_k$ with $k\in[m]$ such that $\varphi(u_k,u_kv_i)=\ell=i+k$~(mod $m$). 
Since $k\in[m]$, $k=\ell-i$~(mod $m$). This implies that the vertex $u_k$ is uniquely determined.
Hence the color $\ell$ appears exactly once among $A_{v_i}$.

On the other hand, for each vertex $u_j$ with $j\in [m]$ and each color $\ell\in \varphi(I_{u_j})$, there is a vertex $v_i$  with $i\in [n]$ such that
$\varphi(u_j,u_jv_i)=\ell=i+j$~(mod~$m$) by the construction of $\varphi$. If $\ell\in \varphi(A_{u_j})$, then by \ref{b}, there is a vertex $v_k$ with $k\in [n]\setminus \{i\}$ such that $\varphi(v_k,u_jv_k)=\ell=k+j-1$~(mod~$m$). This implies $k-1=i$~(mod~$m$). 
Since $1\leq k\leq n\leq m$, $k=i+1$~(mod~$m$). 
Hence the vertex $u_k$ is uniquely determined if it exists (note that it may happen that $i=m-1$ and $n<m$, and thus there is no solution for $k$).
Hence the color $\ell$  appears at most once among $A_{u_j}$.

Therefore, \ref{c} holds in each case, and thus $m\geq \chi^d(K_{m,n})\geq \Delta(K_{m,n})=m$ for every integer $d\geq 1$.
\end{proof}

Note that the proof of Theorem \ref{kmn} naturally yields a linear-time algorithm to construct a $1$-defective incidence coloring of $K_{m,n}$ using $\Delta(K_{m,n})$ colors. 

We conclude the following from Theorems \ref{thm:path}, \ref{thm:tree}, and \ref{kmn}. 
\begin{cor}
$\defi(G)=1$ if $G$ is a path, a cycle, a tree, or a bipartite graph.
\end{cor}

\begin{algorithm}[htp]
\BlankLine
\KwIn{Integers $m$ and $n$ with $m\geq n$ ;}
\KwOut{A $1$-defective incidence coloring of $K_{m,n}$ using colors $m$.}
\BlankLine
\tcc{ $\{u_1,u_2,\ldots u_m\}$ and $\{v_1,v_2,\ldots v_n\}$ are the bipartition of $K_{m,n}$.}
\For{$i=1$ to $n$}
{\For{$j=1$ to $m$}
{$\varphi(v_i,u_jv_i) \gets\ i+j-1$ (mod $m$)\\
 $\varphi(u_j,u_jv_i) \gets\ i+j$ (mod $m$)}
 }
\caption{$1$-defective  incidence coloring of $K_{m,n}$}
\end{algorithm}

\section{Complete graphs and Latin squares} \label{sec:complete}

\begin{lem}\label{keylemma}
$\varphi$ is a $1$-defective incidence $(n-1)$-coloring of $K_n$ if and only if every two strong incidences and every two weak incidences of any vertex are colored differently under $\varphi$, and  $(u,uv)$ and $(v,uv)$ receive distinct colors for each pair of vertices $u$ and $v$.
\end{lem}

\begin{proof}
The sufficiency is obvious since \ref{a}, \ref{b}, and \ref{c} with $d=1$ hold trivially. 
For the necessary, we know that every two strong incidences of any vertex are colored differently under $\varphi$ by \ref{a}, and 
$(u,uv)$ and $(v,uv)$ receive distinct colors for each pair of vertices $u$ and $v$ by \ref{b}. For each vertex $v$, since
$\varphi(I_v)=[\Delta]$, there is no color appearing twice on $A_v$ by \ref{c} with $d=1$. Since $|A_v|=n-1$ and $\varphi$ has exactly $n-1$ colors, every two weak incidences of $v$ are colored differently under $\varphi$.
\end{proof}

A \textit{Latin square} is an $n \times n$ square matrix whose entries consist of $n$ symbols such that each symbol appears exactly once in each row
and each column. An \textit{intercalate} in a Latin square $L$ is a $2 \times 2$ Latin subsquare, that is, two rows and two columns whose intersection includes only two symbols. A \textit{principal  intercalate} in a Latin square $L$ is an intercalate obtained by striking out from $L$ the same rows as columns.

\begin{lem}\label{keylemma2}
$K_n$ has a $1$-defective incidence $(n-1)$-coloring if and only if there is an $n \times n$ Latin square without principal intercalates whose 
entries at the main diagonal are the same.
\end{lem}

\begin{proof}
Let $L$ be an $n \times n$ Latin square without principal intercalates whose 
entries on the main diagonal are the same. 
Permute (if necessary) the symbols $0,1,\ldots,n-1$ of $L$ so that the entries at the main diagonal are $0$.
Clearly, the resulting Latin square $\widetilde{L}$ still has no principal intercalates.  

Let $v_1,v_2,\ldots,v_n$ be vertices of $K_n$. We color the incidence $(v_i,v_iv_j)$ with the entry $\widetilde{L}(i,j)$ of the Latin square that appears at row $i$ and column $j$, and denote this coloring by $\varphi$. Now for each vertex $v_i$, $\varphi(I_{v_i})=\{\widetilde{L}(i,j)~|~j\in [n]\setminus \{i\}\}=[n-1]$ and
$\varphi(A_{v_i})=\{\widetilde{L}(j,i)~|~j\in [n]\setminus \{i\}\}=[n-1]$, since $\widetilde{L}$ is a Latin square with $\widetilde{L}(i,i)=0$ for each $i\in [n]$.
This implies that every two strong incidences and every two weak incidences of $v_i$ are colored differently under $\varphi$.
For each pair of vertices $v_i$ and $v_j$, $\varphi(v_i,v_iv_j)=\widetilde{L}(i,j)$ and $\varphi(v_j,v_iv_j)=\widetilde{L}(j,i)$. Since $\widetilde{L}(i,i)=\widetilde{L}(j,j)=0$ and $\widetilde{L}$ has no 
principal intercalates, $\widetilde{L}(i,j)\neq \widetilde{L}(j,i)$, and thus $\varphi(v_i,v_iv_j)\neq \varphi(v_j,v_iv_j)$. By Lemma \ref{keylemma}, $\varphi$ is a $1$-defective incidence $(n-1)$-coloring of $K_n$.

In the other direction, if $\varphi$ is a $1$-defective incidence $n$-coloring of $K_n$, then we construct an $n \times n$ Latin square $L$ whose entry $L(i,j)$ at row $i$ and column $j$ with $i\neq j$ is $\varphi(v_i,v_iv_j)$, and entries at the main diagonal are all $n$. Since for each pair of $i$ and $j$, $L(i,j)=\varphi(v_i,v_iv_j)\neq \varphi(v_j,v_iv_j)=L(j,i)$ and $L(i,i)=L(j,j)=0$, $L$ has no principal intercalates.
\end{proof}

\begin{lem}\cite{McLeish,Kotzig1976263,KLR}\label{latin}
There exist $n\times n$ Latin squares with no intercalates if and only if $n\neq 2,4$,
\end{lem}

\begin{lem}\label{intercalate}
A $4\times 4$ Latin square with same entries at the main diagonal 
has no intercalate if and only if it has no principal intercalate.
\end{lem}

\begin{proof}
The necessary is trivial so we prove the sufficiency.
Suppose for a contradiction $L$ has an intercalate appearing at the intersection of rows $i,j$ ($i<j$) and columns $m,n$ ($m<n$).  
Assume that $L(i,m)=L(j,n)=a$ and $L(i,n)=L(j,m)=b$. 
Since this intercalate is not principal, $L(1,1)\not\in \{a,b\}$.
It follows that $i\not\in \{m,n\}$, $j\not\in \{m,n\}$ and thus $\{i,j,m,n\}=[4]$. 
Therefore, $L(i,j)=L(j,i)$ and thus there is a principal intercalate
appearing at the intersection of rows $i,j$ and columns $i,j$, a contradiction.
\end{proof}

\begin{thm}\label{complete}
\begin{align*}
  \chi^1(K_{n})= 
  \begin{cases}
        n-1
        &\text{if } n\neq 2,4,
        \\
        n
        &\text{otherwise}.
        \end{cases}
\end{align*}
\end{thm}

\begin{proof}
First, it is trivial $\chi^1(K_2)=2$.
If $n\neq 2,4$, then there exists an $n\times n$ Latin square $L$ with no intercalates by Lemma \ref{latin}.
 Permute (if necessary) the rows of $L$ so that the entries at the main diagonal are the same. One can easily observe that the resulting Latin square $\widetilde{L}$ does not has any intercalate either, and thus has no principal intercalates.
Hence $K_n$ has a $1$-defective incidence $(n-1)$-coloring by Lemma \ref{keylemma2}. It follows that $\Delta(K_n)\leq \chi^1(K_{n})\leq n-1=\Delta(K_n)$.

If $K_4$ has a $1$-defective incidence $3$-coloring, then by Lemma \ref{keylemma2},
there is a $4 \times 4$ Latin square $L$ without principal intercalates whose 
entries at the main diagonal are the same, contradicting Lemma \ref{intercalate}.
Hence $\chi^1(K_{4})\geq 4$. On the other hand, one can easily construct a $1$-defective incidence $4$-coloring of $K_4$, implying $\chi^1(K_{4})\leq 4$.
\end{proof}

 To construct a $1$-defective incidence coloring of $K_n$ ($n=3$ or $n\geq 5$) with $n-1$ colors, we need by Lemma \ref{keylemma2}  to generate an $n \times n$ Latin square without principal intercalates whose entries at the main diagonal are the same. Lemma \ref{latin} guarantees the existence of such an Latin square. 

In the following, we discuss the recursive construction of such Latin squares and design a quadratic-time algorithm to generate an $n \times n$ Latin square satisfying those properties for every $n\neq 2,4$. 

An $n \times n$ matrix $M=(m_{ij})$ is a \textit{circulant} if it has the form 
$m_{ij}=a_{j-i}$ for some $a_0,a_1,\ldots,a_{n-1}$, where the subscript $j-i$ is taken modulo $n$.
In this paper we denote such a matrix by $(a_0,a_1,\ldots,a_{n-1})_{{\rm circ}}$.

\begin{lem}\label{lem:a}
Every $n\times n$ circulant matrix with $n$ being odd is a Latin square without principal intercalates.
\end{lem}
\begin{proof}
Let $M=(a_0,a_1,\ldots,a_{n-1})_{{\rm circ}}$ where $n$ is odd.
Clearly, $m_{ii}=a_0$ for each $1\leq i\leq n$.  It is sufficient to check that $a_{ij}\neq a_{ji}$ if $i>j$. 
Since $n$ is odd, $j-i\neq n+i-j$. This implies 
$m_{ij}=a_{j-i}\not=a_{n+i-j}=m_{ji}$.
\end{proof}

\begin{lem}\label{lem:b}
For a positive even integer $n$ that is not a power of 2, there exists an odd $m\geq 3$ such that the quotient of n divided by m is a positive power of 2
\end{lem}
\begin{proof}
Let $t$ be the largest integer such that $2^t~|~n$ and let $n=2^t\cdot m$. If $m$ is even then $2^{t+1}~|~n$, contradicting the choice of $t$. If $m=1$, then $n=2^t$, contradicting the choice of $n$. Hence $m\geq 3$ is an odd.
\end{proof}

Let $A$ be an $n\times n$ matrix and let
$$
  A\nabla A=\left[
  \begin{matrix}
  A & A+nJ  \\
  B & A  \\
  \end{matrix}
  \right],
 $$
  where
 $$
  B=(A+nJ)^T\left[
  \begin{matrix}
  0 & E_{n-1}  \\
  1 & 0  \\
  \end{matrix}
  \right]$$
 and $J$ is an $n\times n$ matrix in which each element is $1$.
 
 We use $A^{(0\nabla)}$ and $A^{(1\nabla)}$ to denote $A$ and $A\nabla A$, respectively.
 By $A^{(t\nabla)}$ with integer $t\geq 2$, we denote $A^{((t-1)\nabla)}\nabla A^{((t-1)\nabla)}$. 
 Clearly, $A^{(t\nabla)}$ is a $2^t n \times  2^t n$ matrix. 

\begin{lem}\label{lem:c}
If $A$ is a Latin square without principal intercalates, then so does $A^{(t\nabla)}$ for each $t\geq 1$.
\end{lem}
\begin{proof}
It is sufficient to show that $A\nabla A$ is a Latin square without principal intercalates.

Let $A=(a_{ij})_{n\times n}$ and $A\nabla A=(m_{ij})_{2n\times 2n}$.
According to the construction of $A\nabla A$, 
\begin{align*}
        m_{ij}
        =
        \begin{cases}
        a_{ij}
        &\text{if } 1\leq i \leq n~\text{and }1\leq j \leq n;
        \\
        a_{i,j-n}+n
        &\text{if } 1\leq i\leq n~\text{and }n+1\leq j\leq 2n;
        \\
         a_{n,i-n}+n
        &\text{if } n+1\leq i\leq 2n~\text{and }j=1;
        \\
        a_{j-1,i-n}+n
        &\text{if } n+1\leq i\leq 2n~\text{and }2\leq j\leq n;
        \\
        a_{i-n,j-n}
        &\text{if } n+1\leq i\leq 2n~\text{and }n+1\leq j\leq 2n.
        \end{cases}
   \end{align*}
 To complete the proof, we show that  $m_{ij}\neq m_{ji}$ for each pair of $i$ and $j$ with $i\neq j$.  
 
 If $1\leq i\leq n$ and $1\le j\leq n$, then $m_{ij}=a_{ij}$ and $m_{ji}=a_{ji}$. Since $A$ is a Latin square without principal intercalates, $a_{ij}\neq a_{ji}$. It follows that $m_{ij}\neq m_{ji}$.

 If $1\leq i\leq n$ and $n+1\leq j\leq 2n$, then $m_{ij}=a_{i,j-n}+n$, $m_{ji}=a_{i-1,j-n}+n$ if $i\neq 1$, and $m_{ji}=a_{n,j-n}+n$ if $i=1$. Since $a_{i,j-n}$ and $a_{i-1,j-n}$ (resp.\,$a_{1,j-n}$ and $a_{n,j-n}$) are in the same column and not in the same row of $A$, 
they are not equal and thus $m_{ij}\neq m_{ji}$.
 
If $n+1\leq i \leq 2n$ and $n+1\leq j \leq 2n$, then $m_{ij}=a_{i-n,j-n}\neq a_{j-n,i-n}=m_{ji}$, since $A$ is a Latin square without principal intercalates.
 
Hence in each case we have $m_{ij}\neq m_{ji}$, as desired. 
\end{proof}

\begin{algorithm}[htp]
\BlankLine
\KwIn{An integer $n\neq 2,4$;}
\KwOut{An $n \times n$ Latin square $L$ without principal intercalates  whose entries at the main diagonal are 0.}
\BlankLine

\eIf{$n$ is odd}
{$L \gets (0,1,\ldots,n-1)_{{\rm circ}}$}
{\eIf{$n=2^t$ for some integer $t\geq 3$}
{
$A \gets \begin{bmatrix}
0 & 1 & 2 & 3 & 4 & 5 & 6 & 7\\
2 & 0 & 7 & 4 & 5 & 3 & 1 & 6\\
3 & 6 & 0 & 7 & 2 & 1 & 4 & 5\\
4 & 5 & 6 & 0 & 3 & 2 & 7 & 1\\
7 & 4 & 5 & 1 & 0 & 6 & 3 & 2\\
1 & 2 & 4 & 6 & 7 & 0 & 5 & 3\\
5 & 3 & 1 & 2 & 6 & 7 & 0 & 4\\
6 & 7 & 3 & 5 & 1 & 4 & 2 & 0 
\end{bmatrix}$\\
$L \gets A^{((t-3)\nabla)}$
}{
Find the largest integer $t$ such that $2^t~|~n$ and let $m=n/2^t$\\
$A\gets (0,1,\ldots,m-1)_{{\rm circ}}$\\
$L \gets A^{(t\nabla)}$
}
}
\caption{\texttt{LATIN-SQARE}(n)}\label{algo:b}
\end{algorithm}

\begin{algorithm}[htp]
\BlankLine
\KwIn{An integer $n\neq 2,4$;}
\KwOut{A $1$-defective incidence coloring of $K_n$ using colors from $\{1,2,\ldots,n\}$.}
\BlankLine
\tcc{The set of vertices of $K_n$ is $\{v_1,v_2,\ldots,v_n\}$.}
$L \gets$ \texttt{LATIN-SQARE}(n)\\
\For{$i=1:n$}
{\For{$j=1:n$ and $j\neq i$}
{Color the incidence $(v_i,v_iv_j)$ with $L(i,j)$}}
\caption{\texttt{COLOR-COMPLETE-GRAPH}(n)}\label{algo:c}
\end{algorithm}

The following Algorithm \ref{algo:b} outputs in $O(n^2)$ time an $n \times n$ Latin square $L$ without principal intercalates  whose entries at the main diagonal are 0 whenever we input an integer $n\neq 2,4$. It works by Lemmas \ref{lem:a}, \ref{lem:b}, and \ref{lem:c}. Afterwards, by Lemma \ref{keylemma2}, we can construct a $1$-defective incidence coloring of $K_n$ using colors from $\{1,2,\ldots,n\}$ according to Algorithm \ref{algo:c}, which also runs in $O(n^2)$ time.

To close this section, we present the following.

\begin{thm}\label{complete2}
$\chi^d(K_n)=n-1$ for every integer $d\geq 2$ and $n\neq 2$.
\end{thm}

\begin{proof}
Theorem \ref{complete} and \eqref{rela} imply the result for $n\neq 1,4$. The case of $n=1$ is trivial and 
a $2$-defective $3$-coloring of $K_4$ can be easily constructed. Hence $\chi^d(K_4)=3$ for every integer $d\geq 2$.
\end{proof}

Combining Theorems \ref{complete} with \ref{complete2}, we conclude the following.

\begin{cor}

\begin{align*}
  \defi(K_{n})= 
  \begin{cases}
        1
        &\text{if } n\neq 2, 4,
        \\
        2
        &\text{if } n=4,
        \\
        \infty
        &\text{if } n=2.
        \end{cases}
\end{align*}
\end{cor}
Note that $\chi^d(K_2)=2$ for every integer $d\geq 1$.


\section{Outerplanar graphs} \label{sec:outerplanar}

\begin{defn}\label{def:4}
A \textit{conditional incidence $\Delta$-coloring} of $G$ is 
an incidence $\Delta$-coloring such that \vspace{-2mm}
\begin{enumerate}[label=\textbf{$(\alph*\ref{def:4})$}]\setlength{\itemsep}{-3pt}
\item \label{i} $(u,uv)$ and $(v,uv)$ receive distinct colors for each edge $uv$;
\item \label{ii} each color appears at most once among $A_u$ for each vertex $u$;
\item \label{iii} each color appears at most once among $I_u$ for each vertex $u$;
\item \label{iv} each color appears at least once among $A_u\cup I_u$ for each vertex $u$ with 
$\deg_G(u)\geq \Delta-1$.  \vspace{-2mm}
\end{enumerate}
\end{defn}

\begin{obs}\label{obs}
Any conditional incidence $\Delta$-coloring of $G$ is a $1$-defective incidence coloring. \hfill$\square$ 
\end{obs}

An \textit{outplanar graph} is a graph that can be embedded in the plane in such a way that all vertices lie on the outer face.
In this section we show the following theorem.

\begin{thm}\label{thm:outerplanar}
If $G$ is an outerplanar graph with $\Delta(G)\leq \Delta$ and $\Delta\geq 4$, then $G$ has a conditional incidence $\Delta$-coloring.
\end{thm}

Combining Observation \ref{obs} and Theorem \ref{thm:outerplanar}, we immediately obtain the following.

\begin{thm}\label{thm:outerplanar2}
$\chi^1(G)=\Delta(G)$ if $G$ is an outerplanar graph with $\Delta(G)\geq 4$. \hfill $\square$
\end{thm}

\begin{rem}\label{rem1}
The bound 4 in Theorems \ref{thm:outerplanar} and \ref{thm:outerplanar2} are both sharp.
To see this, we look at the outerplanar graph $H$ derived from a cycle $uxvy$ of length four via adding an edge $uv$ and a pendant edge $xz$. 
It has maximum degree 3. 
Suppose that $H$ has a $1$-defective incidence $3$-coloring $\varphi$.  
Assume by symmetry that $\varphi(u,ux)=1$, $\varphi(u,uv)=2$ and $\varphi(u,uy)=3$. It follows that $\varphi(v,uv)\neq 2$.
If $\varphi(v,uv)=1$, then $\varphi(x,xv)=1$, $\varphi(y,uy)=\varphi(v,vy)=2$, and $\varphi(x,ux)=\varphi(v,vx)=\varphi(y,vy)=3$,
which forces $\varphi(x,xz)=\varphi(z,xz)=2$. 
If $\varphi(v,uv)=3$, then $\varphi(y,uy)=\varphi(v,vy)=\varphi(x,vx)=1$, $\varphi(v,xv)=\varphi(x,ux)=2$, and $\varphi(y,vy)=3$, forcing $\varphi(x,xz)=\varphi(z,xz)=3$. 
Each case contradicts \ref{b} of Definition \ref{def:3}.
Hence there exist outerplanar graphs with maximum degree 3 that are not $1$-defectively incidence $3$-colorable, and thus not conditionally incidence $3$-colorable.
\end{rem}

A graph $H$ is a \textit{minor} of a graph $G$ if a copy of $H$ can be obtained from $G$ via repeated edge deletion and/or edge contraction. 
Chartrand and Harary \cite{AIHPB_1967__3_4_433_0} first pointed out that a graph is outerplanar if and only if it is $\{K_4,K_{2,3}\}$-minor-free.
The following classic structural theorem on outerplanar graphs were applied in quite a lot of papers.

\begin{lem}\cite{WZ99}\label{lem:outerplanar}
Every outerplanar graph $G$ contains one of the following configurations
\begin{enumerate}[label={\rm (C\arabic*)}]\setlength{\itemsep}{-3pt}
\item \label{c1} a vertex of degree at most 1;
\item \label{c2} an edge $uv$ with $\deg_G(u)=\deg_G(v)=2$;
\item \label{c3} a triangle $uvw$ with $\deg_G(u)=2$ and $\deg_G(v)=3$;
\item \label{c4} two intersecting triangle $uvx$ and $uwy$ such that $\deg_G(u)=4$ and $\deg_G(v)=\deg(w)=2$.
\end{enumerate}
\end{lem}

In the following, we show a series of self-contained lemmas that not only support the proof of Theorem \ref{thm:outerplanar} but also can be applied in other places.

A family $\mathcal{G}$ of graphs is \textit{hereditary} (resp.\,\textit{minor-closed}) if every subgraph (resp.\,minor) $H$ of each graph $G\in \mathcal{G}$ belongs to $\mathcal{G}$. 
A graph $G$ is \textit{$\rho$-$\Delta$-critical} (resp\,\textit{$\varrho$-$\Delta$-critical}) if $\Delta(G)\leq \Delta$ and
$G$ has no conditional incidence $\Delta$-coloring but every subgraph  (resp.\,minor) $H$ of $G$ does.
Since every subgraph is a minor, we have the following.

\begin{lem}\label{lem:relationship}
A $\varrho$-$\Delta$-critical graph is definitely $\rho$-$\Delta$-critical. \hfill$\square$
\end{lem}

\begin{lem}\label{lem:connected}
If $G$ is a $\rho$-$\Delta$-critical graph, then $G$ is connected.
\end{lem}

\begin{proof}
If $G$ has two components $G_1$ and $G_2$, then $G_1$ and $G_2$ are conditionally incidence $\Delta$-colorable. Combining the 
conditional incidence $\Delta$-colorings of them, we obtain a conditional incidence $\Delta$-coloring of $G$, contradicting the fact that $G$ is 
$\rho$-$\Delta$-critical.
\end{proof}

\begin{lem}\label{lem:minimumdegree2}
If $G$ is a $\rho$-$\Delta$-critical graph, then $\delta(G)\geq 2$.
\end{lem}

\begin{proof}
Suppose for a contradiction that $G$ has a vertex $v$ of degree at most 1. By Lemma \ref{lem:connected}, we may assume $N_G(v)=\{u\}$.
Since $G$ is $\rho$-$\Delta$-critical, $G'=G-v$ admits a conditional incidence $\Delta$-coloring $\varphi'$.

We extend $\varphi'$ to a conditional incidence $\Delta$-coloring $\varphi$ of $G$ as follows.


If $\varphi'(I_u\cup A_u)=[\Delta]$, then $\varphi'(I_u)\setminus \varphi'(A_u)\neq \emptyset$ and $\varphi'(A_u)\setminus \varphi'(I_u)\neq \emptyset$, since
$\deg_{G'}(u)\leq \Delta(G)-1\leq \Delta-1$.
Hence we can color $(u,uv)$ with $a\in \varphi'(A_u)\setminus \varphi'(I_u)$ and $(v,uv)$ with $b\in \varphi'(I_u)\setminus \varphi'(A_u)$ to complete $\varphi$.

If $|\varphi'(I_u\cup A_u)|<\Delta$, then $\deg_{G'}(u)\leq \Delta-2$.
If $\deg_{G'}(u)< \Delta-2$, then color $(u,uv)$ with $a\in [\Delta]\setminus \varphi'(I_u)$.
and $(v,uv)$ with $b\in [\Delta]\setminus (\varphi'(A_u)\cup \{a\})$.
If $\deg_{G'}(u)= \Delta-2$, then $[\Delta]\setminus \varphi'(I_u)=\{a,b\}$ has one element, say $b$, that does not appear in $\varphi'(A_u)$.
Hence we can color $(u,uv)$ with $a$ and $(v,uv)$ with $b$ to complete $\varphi$. Note that this extension guarantees $\varphi(I_u\cup A_u)=[\Delta]$.
\end{proof}

\begin{rem} If $G$ is an outerplanar graph with $\Delta(G)\leq \Delta$ such that $G$ has no 1-defective incidence $\Delta$-coloring but every subgraph $H$ of $G$ does, then the idea of proving Lemma \ref{lem:minimumdegree2} cannot be applied to prove $\delta(G)\geq 2$. 
Actually, if $G$ has an edge $uv$ with $\deg_G(v)=1$ and $\deg_G(u)=\Delta$, then $G-v$ has a 1-defective incidence $\Delta$-coloring $\varphi'$ by the minimality of $G$. However, it may happen that $\varphi'(I_u)=\varphi'(A_u)=[\Delta-1]$ and thus we have to color both $(u,uv)$ and $(v,uv)$ with $\Delta$ while extending $\varphi'$ to $G$ and then return a failure. This is indeed the reason why we introduce the notion of conditional incidence coloring and then prove Theorem \ref{thm:outerplanar} instead of proving Theorem \ref{thm:outerplanar2} directly.
\end{rem}

\begin{lem}\label{lem:2-no-minus2}
If $G$ is a $\rho$-$\Delta$-critical graph, then $G$ does not contain an edge $uv$ with $\deg_G(u)=2$ and $\deg_G(v)\leq \Delta-2$.
\end{lem}

\begin{proof}
 Suppose for a contradiction that $G$ has such an edge $uv$.
By Lemma \ref{lem:minimumdegree2}, we have $2\leq \deg_G(v)\leq \Delta-2$ and thus $\Delta\geq 4$. 
Let $w$ be the other neighbor of $u$ besides $v$. Since $G$ is a $\rho$-$\Delta$-critical, $G'=G-uv$ admits a conditional incidence $\Delta$-coloring $\varphi'$.

Assume $\varphi(w,uw)=a$ and $\varphi(u,uw)=b$.
Since $\deg_{G'}(v)\leq (\Delta-2)-1 \leq \Delta-3$, 
we are able to color $(v,uv)$ with a color $a'\in [\Delta]\setminus (\varphi'(I_v)\cup \{a\})$ and
$(u,uv)$ with a color from $[\Delta]\setminus (\varphi'(A_v) \cup \{a',b\})$.
This extends $\varphi'$ to a conditional incidence $\Delta$-coloring of $G$.
\end{proof}

\begin{lem}\label{lem:2-no-twominus1}
If $G$ is a $\rho$-$\Delta$-critical graph and $\Delta\geq 4$, then $G$ does not contain a vertex with $\deg_G(u)=2$ such that $N_G(u)=\{v,w\}$ and $\deg_G(v)=\deg_G(w)=\Delta-1$.
\end{lem}
\begin{proof}
Suppose for a contradiction that $G$ contains such a vertex. 
Since $G$ is $\rho$-$\Delta$-critical, $G'=G-uw$ has a conditional incidence $\Delta$-coloring $\varphi'$.
Assume $\varphi'(v,uv)=a$, $\varphi'(u,uv)=b$, $[\Delta]\setminus\varphi'(I_w)=\{a',c'\}$, and $[\Delta]\setminus\varphi'(A_w)=\{b',d'\}$.
We extend $\varphi'$ to an incidence $\Delta$-coloring $\varphi$ as follows.

Case 1. $\{a',c'\}\cap \{b',d'\}=\emptyset$

Since $\deg_{G'}(w)=\Delta-2$, $\varphi'(I_w\cup A_w)=[\Delta]$ under this case. 
So we color $(w,uw)$ with a color in $\{a',c'\}\setminus \{a\}$ and $(u,uw)$ with a color in $\{b',d'\}\setminus \{b\}$ to complete $\varphi$.
It is easy to see that $\varphi$ is a conditional incidence $\Delta$-coloring of $G$. 

Case 2. $a\in \{a',c'\}\cap \{b',d'\}$.

Assume, without loss of generality, that  $a'=b'=a$.
Coloring $(u,uw)$ with $a$ and $(w,uw)$ with $c'$, we obtain a conditional incidence $\Delta$-coloring $\varphi$ of $G$  as
$\varphi(I_w\cup A_w)\supseteq \varphi'(I_w) \cup \{a,c'\}=\varphi'(I_w) \cup \{a',c'\}=[\Delta]$,
$\varphi(u,uw)=a\neq b=\varphi(u,uv)$, and $\varphi(w,uw)=c'\neq a'=a=\varphi(v,uv)$.

Case 3. $a\not\in \{a',c'\}\cap \{b',d'\}\neq \emptyset$.

Assume, without loss of generality, that $a'=b'\neq a$.

Subcase 3.1. $b\neq d'$.

We color $(w,uw)$ with $a'$ and $(u,uw)$ with $d'$. 
This extended coloring $\varphi$ of $G$  satisfies $\varphi(I_w\cup A_w)\supseteq \varphi'(A_w) \cup \{a',d'\}=\varphi'(A_w) \cup \{b',d'\}=[\Delta]$, $\varphi(w,uw)=a'\neq a=\varphi(v,uv)$, and $\varphi(u,uw)=d'\neq b=\varphi(u,uv)$, and thus is a conditional incidence $\Delta$-coloring of $G$.

Subcase 3.2. $a\neq c'$.

We color $(w,uw)$ with $c'$ and $(u,uw)$ with $b'$. 
This extended coloring $\varphi$ of $G$ satisfies $\varphi(I_w\cup A_w)\supseteq \varphi'(A_w) \cup \{b',c'\}=\varphi'(A_w) \cup \{a',c'\}=[\Delta]$, $\varphi(w,uw)=c'\neq a=\varphi(v,uv)$, and $\varphi(u,uw)=b'\neq b=\varphi(u,uv)$. It follows that $\varphi$ is a conditional incidence $\Delta$-coloring of $G$. 

Subcase 3.3. $b=d'$ and $a=c'$.

This is equivalent to say $[\Delta]\setminus\varphi'(I_w)=\{a,a'\}$ and $[\Delta]\setminus\varphi'(A_w)=\{b,a'\}$.

Erase the colors of $(u,uv)$ and $(v,uv)$. Now we have two ways to transfer $\varphi'$ to a conditional incidence $\Delta$-coloring of $G''=G-uv$.
The first way is to color $(w,uw)$ with $a'$ and $(u,uw)$ with $b$, while the second way is to color $(w,uw)$ with $a$ and $(u,uw)$ with $a'$.
We denote by $\varphi_1$ and $\varphi_2$ be those two colorings respectively.

Assume $[\Delta]\setminus\varphi_1(I_v)=\{a,c''\}$ and $[\Delta]\setminus\varphi_1(A_v)=\{b,d''\}$.
Using the same proof strategies in Cases 1 and 2 and Subcases 3.1 and 3.2, one can find that the only obstacle while extending $\varphi_1$ to a conditional incidence $\Delta$-coloring of $G$ is the case that 
$[\Delta]\setminus\varphi_1(I_v)=\{a',e\}$ and $[\Delta]\setminus\varphi_1(A_v)=\{b,e\}$ for some $e\in [\Delta]$.
In particular, since $a\neq a'$, we deduce $e=d''=a$ and $a'=c''$. 

Therefore, 
$[\Delta]\setminus\varphi_2(I_v)=[\Delta]\setminus\varphi_1(I_v)=\{a',a\}$ and 
$[\Delta]\setminus\varphi_2(A_v)=[\Delta]\setminus\varphi_1(A_v)=\{a,b\}$
if $\varphi_1$ is not extendable. However, we earn chance to extend $\varphi_2$ now.
This can be done by coloring $(u,uv)$ with $a$ and $(v,uv)$ with $a'$.
Since the resulting coloring $\varphi$ satisfies $\varphi(I_v\cup A_v)\supseteq \varphi_2(I_v)\cup \{a,a'\}=[\Delta]$, $\varphi(u,uv)=a\neq a'=\varphi_2(u,uw)$, and $\varphi(v,uv)=a'\neq a=\varphi_2(w,uw)$, $\varphi$ is
a  conditional incidence $\Delta$-coloring of $G$.
\end{proof}

\begin{lem}\label{lem:no234triangle}
If $G$ is a $\rho$-$\Delta$-critical graph and $\Delta\geq 4$, 
then $G$ does not contain a triangle $uvw$ with $\deg_G(u)=2$, $\deg_G(v)=3$ and $\deg_G(w)=4$.
\end{lem}

\begin{proof}
Suppose for a contradiction that such a triangle exists in $G$. 
By Lemma \ref{lem:2-no-minus2}, we have $\Delta=4$.
Since $G$ is $\rho$-$4$-critical, $G-uv$ has a conditional incidence $4$-coloring $\varphi'$. Let $N_{G}(v)=\{u,w,x\}$ and $N_G(w)=\{u,v,y,z\}$. Assume the colors on $(v,vx)$, $(x,vx)$, $(v,vw)$, $(w,vw)$, $(w,uw)$ and $(u,uw)$ are $1,a,2,b,c,d$, ($a\neq b$, $a\neq1$, $b\neq 2$, $b\neq c$, $c\neq d$ and $d\neq 2$) respectively. 

If $|\{1,2\}\cup \{a,b\}|=2$, then $a=2$ and $b=1$. If $d\not\in \{1,2\}$, say, $d=3$, then $c\neq 3$ and thus we can color $(u,uv)$ with $4$ and $(v,uv)$ with $3$. If $d \in \{1,2\}$, then $d=1$ and color $(v,uv)$ with $a'\in \{3,4\}\setminus \{c\}$ and $(u,uv)$ with $b'\in \{3,4\}\setminus \{a'\}$.

If $|\{1,2\}\cup \{a,b\}|=3$, then assume, without loss of generality, that $4\not\in \{1,2\}\cup \{a,b\}$.

If $a=3$, then $b=1$. 
If $c=4$, then $d\neq 4$, and we color $(u,uv)$ with 4 and $(v,uv)$ with color $3$.
If $c\neq 4$, then we are able to color $(u,uv)$ with 2 as $d\neq 2$ and $(v,uv)$ with color $4$.


If $b=3$, then $a=2$.
If $c=4$, then $d\neq 4$, and we color $(u,uv)$ with 4 and $(v,uv)$ with  color $3$.
If $c\neq 4$ and $d\neq 1$, then we color $(u,uv)$ with 1 and $(v,uv)$ with  color $4$.
If $c\neq 4$ and $d=1$, then $c=2$ and we can color $(u,uv)$ with $4$ and $(v,uv)$ with $3$.



If $|\{1,2\}\cup \{a,b\}|=4$, then assume by symmetry that $a=3$ and $b=4$. Since $d\neq 2$, we are able to  color $(u,uv)$ with $2$ and $(v,uv)$ with a color
in $\{3,4\}\setminus \{c\}$.

In each of the above cases, the resulting extended coloring is a conditional incidence $4$-coloring $\varphi$ of $G$ as the property that $\varphi(I_v\cup A_v)=[4]$ is satisfied throughout the extension.
\end{proof}

\begin{figure}
    \centering
    \includegraphics[width=16cm]{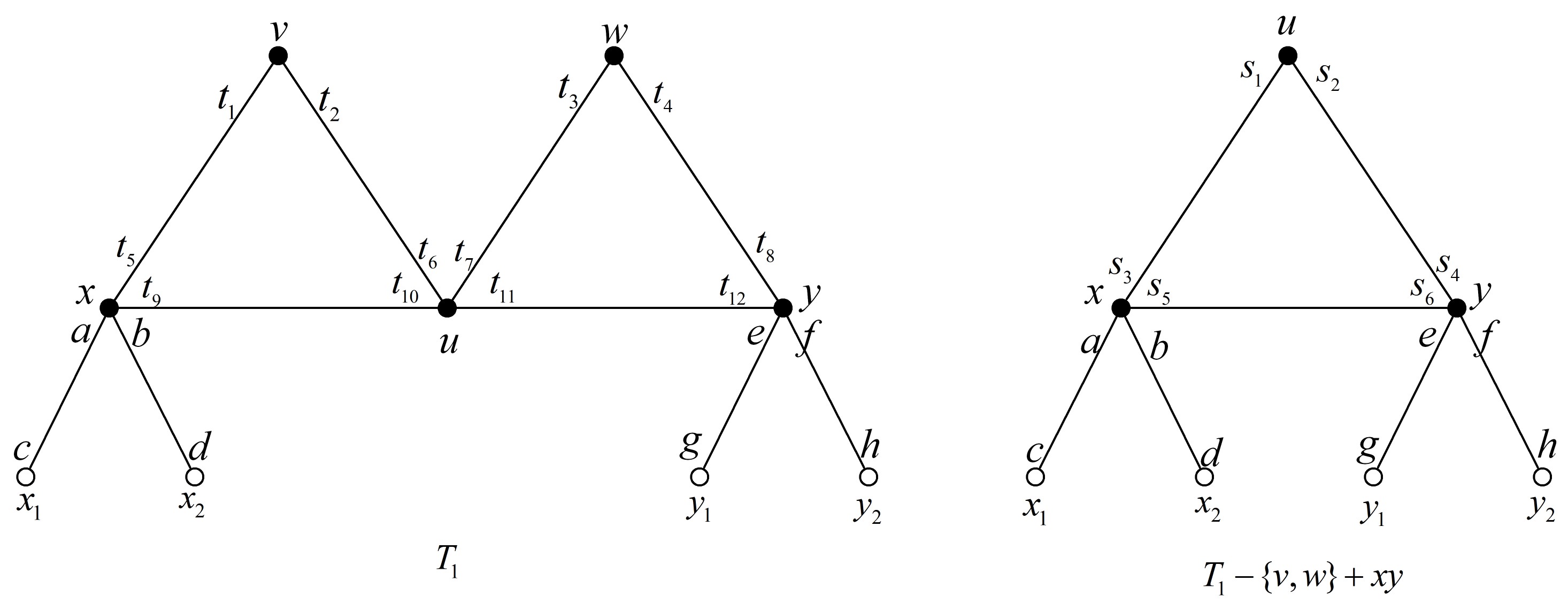}
    \caption{The configuration $T_1$, where $xy$ is a non-edge and we may have $\{x_1,x_2\}\cap \{y_1,y_2\}\neq \emptyset$, and the operation $T_1-\{v,w\}+xy$.}
    \label{fig:pic1}
\end{figure}

In the following, we use $T_1$ to stand for the configuration as shown by the left picture of Figure \ref{fig:pic1}.

\begin{lem}\label{no-T1}
If $G$ is a $\varrho$-$\Delta$-critical graph and $\Delta\geq 4$, then $G$ does not contain the configuration $T_1$.
\end{lem}

\begin{proof}
Suppose to the contrary that $G$ contains a copy of $T_1$. 
By Lemmas \ref{lem:2-no-minus2} and \ref{lem:2-no-twominus1}, we have $\Delta=4$.
Since $G$ is $\varrho$-$4$-critical and $G'=G-\{v,w\}+xy$ is a minor of $G$, $G'$
 has a conditional incidence $4$-coloring $\varphi'$.
The right picture of Figure \ref{fig:pic1} shows a partial coloring of $\varphi'$. 

The idea of the proof is to restrict $\varphi'$ to a partial coloring of $G-\{u,v,w\}$ and then extend it to $G$ by coloring the remaining 12 incidences using colors $t_1,t_2,\ldots,t_{12}$ as shown by the first picture of Figure \ref{fig:pic1}. In order to make the extended coloring being a conditional incidence $4$-coloring of $G$, we need choose each $t_i$ carefully from $[4]$ such that 
\begin{align*}
   \Omega=\prod \limits_{i=1}^{12}\prod \limits_{\lambda\in \Lambda_i}(t_i-\lambda)\neq 0,
\end{align*}
where
$\Lambda_1=\{t_2,t_5,t_{10},c,d\}$,
$\Lambda_2=\{t_3,t_6,t_{9},t_{12}\}$,
$\Lambda_3=\{t_4,t_7,t_{9},t_{12}\}$,
$\Lambda_4=\{t_8,t_{11},g,h\}$,
$\Lambda_5=\{t_6,t_{9},a,b\}$,
$\Lambda_6=\{t_7,t_{10},t_{11}\}$,
$\Lambda_7=\{t_8,t_{10},t_{11}\}$,
$\Lambda_8=\{t_{12},e,f\}$,
$\Lambda_9=\{t_{10},t_{12},a,b\}$,
$\Lambda_{10}=\{t_{11},c,d\}$,
$\Lambda_{11}=\{t_{12},g,h\}$,
and $\lambda_{12}=\{e,f\}$.
If such $t_i$'s exist, then we say that $(a,b,c,d,e,f,g,h)$ is extendable.

We traverse all cases of $(a,b,c,d,e,f,g,h)$ and check whether all of them are extendable by computer assistance.
Algorithm \ref{algo:check} defined a function \texttt{CHECK}($a$,$b$,$c$,$d$,$e$,$f$,$g$,$h$). 
This function returns a non-zero vector if and only if we input integers $a,b,c,d,e,f,g,h\in [4]$ such that $(a,b,c,d,e,f,g,h)$ is extendable.

Assume $a=1$ and $b=2$. We run Algorithm \ref{algo:reducibility} on a usual personal computer using MATLAB. It returns a zero matrix $M$ in less than one minute.
Since $(a,b,c,d,e,f,g,h)$ comes from the conditional incidence $4$-coloring $\varphi'$, it naturally holds that 
$c\not\in \{1,d\}$, $d\neq 2$, $g\not\in \{e,h\}$, and $f\not\in \{e,h\}$. So lines \ref{li:1}--\ref{li:10} of Algorithm \ref{algo:reducibility}
returns a matrix such that if $(a,b,c,d,e,f,g,h)$ is not extendable then it is coincide with some row of this matrix. 
Since Algorithm \ref{algo:reducibility} finally returns a zero matrix, we conclude that 
if $(a,b,c,d,e,f,g,h)$ is not extendable then either 
\begin{itemize}
\item $\{1,2\}=\{e,f\}$, $|\{c,d\}\cap \{g,h\}|=1$ and $(\{c,d\}\cap \{g,h\})\cap \{1,2\}=\emptyset$, or 
\item $\{1,2\}=\{e,f\}$, $|\{c,d\}\cap \{g,h\}|=1$, $(\{c,d\}\cap \{g,h\})\cap \{1,2\}\not=\emptyset$ and $(\{c,d\}\oplus \{g,h\}) \cap \{1,2\}=\emptyset$, or 
\item $\{c,d\}=\{g,h\}$, $|\{1,2\}\cap \{e,f\}|=1$ and $(\{1,2\}\cap \{e,f\})\cap \{c,d\}=\emptyset$, or  
\item $\{c,d\}=\{g,h\}$, $|\{1,2\}\cap \{e,f\}|=1$, $(\{1,2\}\cap \{e,f\})\cap \{c,d\}\not=\emptyset$ and $(\{1,2\}\oplus \{e,f\}) \cap \{c,d\}=\emptyset$.
\end{itemize}
Here $\oplus$ is the operation of symmetric difference.

We look back at the conditional incidence $4$-coloring $\varphi'$. 

If the first case occurs, then assume, without loss of generality, that $e=1$, $f=2$, and $c=g=3$. 
It follows $3\not\in \{s_5,s_6\}$ and thus $s_5=s_6=4$, a contradiction.

If the second case occurs, then assume, without loss of generality, that $e=d=h=1$, $f=2$, $c=3$, and $g=4$. It follows $s_1=s_2=2$, a contradiction.

Similarly, we still have contradiction if we meet the third or the fourth case.

Hence every valid $(a,b,c,d,e,f,g,h)$ is extendable.
\end{proof}

\begin{algorithm}[htp]
\BlankLine
\KwIn{Integers $a,b,c,d,e,f,g,h\in [4]$}
\KwOut{A solution matrix $A$}
\BlankLine

\For{$t_1=1$ to $4$}{
\For{$t_2=1$ to $4$}{
\For{$t_3=1$ to $4$}{
\For{$t_4=1$ to $4$}{
\For{$t_5=1$ to $4$}{
\For{$t_6=1$ to $4$}{
\For{$t_7=1$ to $4$}{
\For{$t_8=1$ to $4$}{
\For{$t_9=1$ to $4$}{
\For{$t_{10}=1$ to $4$}{
\For{$t_{11}=1$ to $4$}{
\For{$t_{12}=1$ to $4$}{
\If{$\Omega\neq 0$}{$A \gets [t_1,t_2,t_3,t_4,t_5,t_6,t_7,t_8,t_9,t_{10},t_{11},t_{12}]$\\
return
}
}
}
}
}
}
}
}
}
}
}
}
}
$A \gets [0,0,0,0,0,0,0,0,0,0,0,0]$
\caption{\texttt{CHECK}($a$,$b$,$c$,$d$,$e$,$f$,$g$,$h$)}\label{algo:check}
\end{algorithm}

\begin{algorithm}[htp]
\BlankLine
\KwIn{Null}
\KwOut{An inspection matrix $M$}
\BlankLine
$i \gets 1$ \label{li:1} \\
\For{$c=1$ to $4$}{
\For{$d=1$ to $4$}{
\For{$e=1$ to $4$}{
\For{$f=1$ to $4$}{
\For{$g=1$ to $4$}{
\For{$h=1$ to $4$}{
\If{$c\not\in \{1,d\}$, $d\neq 2$, $g\not\in \{e,h\}$, $f\not\in \{e,h\}$, and 
\texttt{CHECK}$(1,2,c,d,e,f,g,h)=[0, 0, 0, 0, 0, 0, 0, 0, 0, 0, 0, 0]$}
{
$M(i,:)=[1,2,c,d,e,f,g,h]$\\
$i=i+1$ \label{li:10}
}
}
}
}
}
}
}
$s$ $\gets$ the number of rows of $M$\\
\For{$j=1$ to $s$}
{
$c \gets M(j,3)$\\
$d \gets M(j,4)$\\
$e \gets M(j,5)$\\
$f \gets M(j,6)$\\
$g \gets M(j,7)$\\
$h \gets M(j,8)$\\
\If{$\{1,2\}=\{e,f\}$ and $|\{c,d\}\cap \{g,h\}|=1$}
{\If{$(\{c,d\}\cap \{g,h\})\cap \{1,2\}=\emptyset$ or  $(\{c,d\}\oplus \{g,h\}) \cap \{1,2\}=\emptyset$}{
$M(j,:) \gets [0,0,0,0,0,0,0,0]$
}
}

\If{$\{c,d\}=\{g,h\}$ and $|\{1,2\}\cap \{e,f\}|=1$}
{\If{$(\{1,2\}\cap \{e,f\})\cap \{c,d\}=\emptyset$ or  $(\{1,2\}\oplus \{e,f\}) \cap \{c,d\}=\emptyset$}{
$M(j,:) \gets [0,0,0,0,0,0,0,0]$
}
}
}
\caption{\texttt{REDUCIBILITY-INSPECTION}()}\label{algo:reducibility}
\end{algorithm}

\begin{figure}
    \centering
    \includegraphics[width=8cm]{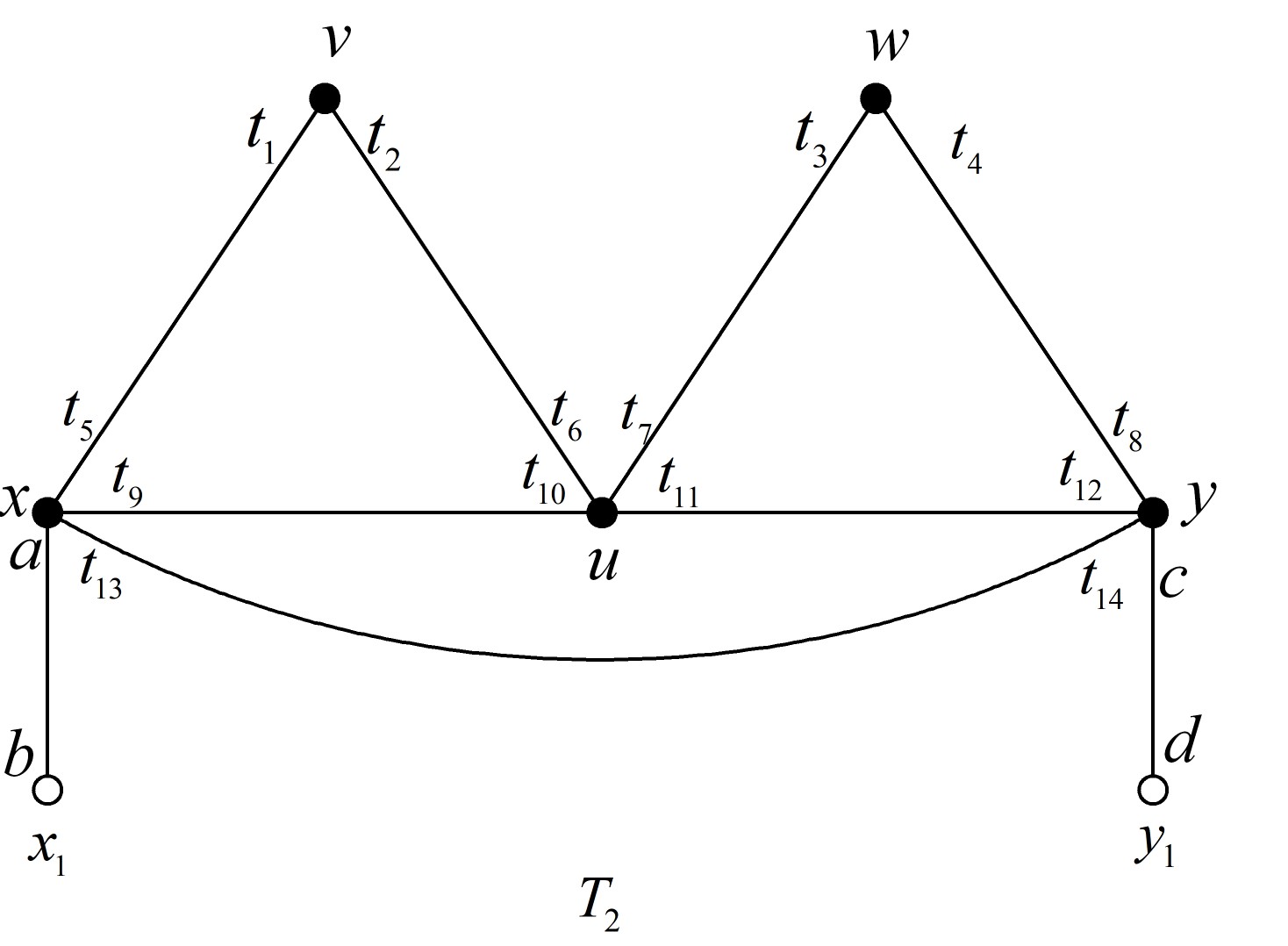}
    \caption{The configuration $T_2$, where we may have $x_1=y_1$.}
    \label{fig:pic2}
\end{figure}

We denote by $T_2$ the configuration described by the picture of Figure \ref{fig:pic2}.

\begin{lem}\label{no-T2}
If $G$ is a $\rho$-$\Delta$-critical graph and $\Delta\geq 4$, then $G$ does not contain the configuration $T_2$.
\end{lem}

\begin{proof}
Suppose for a contradiction that $G$ contains a copy of $T_2$. 
By Lemmas \ref{lem:2-no-minus2} and \ref{lem:2-no-twominus1}, we have $\Delta=4$.
Since $G$ is $\rho$-$4$-critical, $G-\{u,v,w\}$ has a conditional incidence $4$-coloring $\varphi'$. Assume $\varphi'(x,xx_1)=1$, $\varphi'(x_1,xx_1)=2$, $\varphi'(y,yy_1)=c$, and $\varphi(y_1,yy_1)=d$.
We are to extend $\varphi'$ to $G$ by coloring the remaining 14 incidences as marked in the picture with colors $t_1,t_2,\ldots,t_{14}$. Below we distinguish six non-isomorphic cases and show what those $t_i$'s are.
Let $T=[t_1,t_2,\ldots,t_{14}]$.

If $c=1$ and $d=2$, then choose $T=[3,1,4,1,4,2,3,2,2,1,4,3,3,4]$.

If $c=1$ and $d=3$, then choose $T=[3,1,4,1,4,3,2,3,3,1,4,2,2,4]$.

If $c=2$ and $d=1$, then choose $T=[3,4,3,4,4,3,4,3,2,1,2,1,3,4]$.

If $c=2$ and $d=3$, then choose $T=[1,2,1,4,4,3,2,1,3,4,1,4,2,3]$.

If $c=3$ and $d=2$, then choose $T=[1,3,1,4,2,4,2,1,4,3,1,2,3,4]$.

If $c=3$ and $d=4$, then choose $T=[3,4,2,1,4,2,4,2,3,1,3,1,2,4]$.

One can check that the resulting coloring in each case is a conditional incidence $\Delta$-coloring of $G$.
\end{proof}

\begin{figure}
    \centering
    \includegraphics[width=8cm]{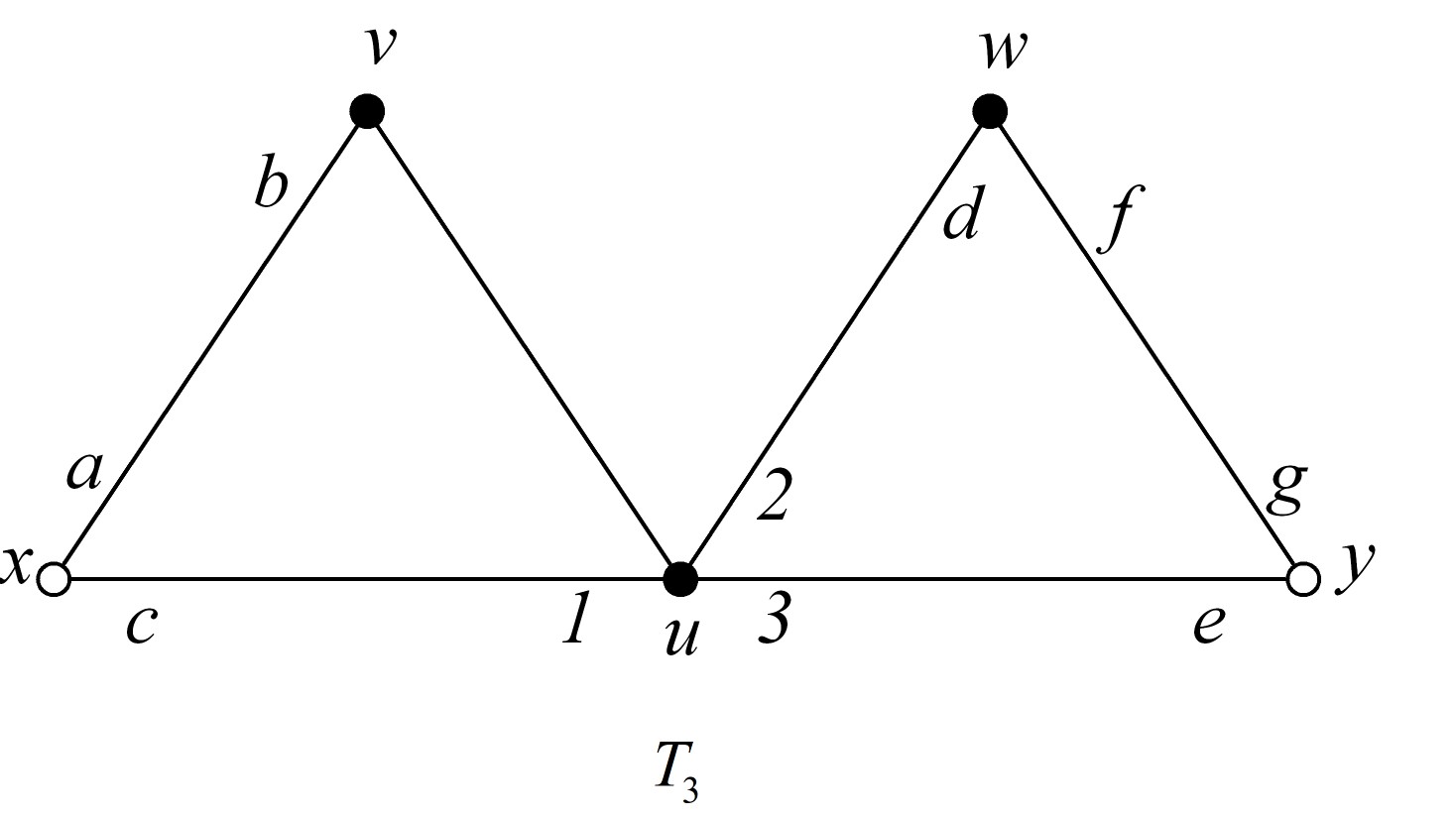}
    \caption{The configuration $T_3$, where $\max\{\deg_G(x),\deg_G(y)\}\leq 5$.}
    \label{fig:pic3}
\end{figure}

We denote by $T_3$ the configuration described by the picture of Figure \ref{fig:pic3}.

\begin{lem}\label{no-T3}
If $G$ is a $\rho$-$\Delta$-critical graph and $\Delta\geq 5$,then $G$ does not contain the configuration $T_3$.
\end{lem}

\begin{proof}
Suppose for a contradiction that $G$ contains a copy of $T_3$.
By Lemma \ref{lem:2-no-minus2}, we have $\Delta=5$.
Since $G$ is a $\rho$-$5$-critical graph, $G-uv$ has a conditional incidence $5$-coloring $\varphi'$.
Figure \ref{fig:pic3} shows a partial coloring of $\varphi'$.

If $\{c,d,e\}=\{1,2,3\}$, then we can color $(u,uv)$ with $a'\in \{4,5\}\setminus \{a\}$
and $(v,uv)$ with $b'\in \{4,5\}\setminus \{a'\}$ to complete a conditional incidence $5$-coloring of $G$ provided $b\not\in \{4,5\}$. 
Hence $b\in \{4,5\}$. Assume by symmetry that $b=4$. It follows that $a\neq 4$ and thus we are able to color $(u,uv)$ with 4 and $(v,uv)$ with 5 to finish a conditional incidence $5$-coloring of $G$.


If $|\{1,2,3\}\cup\{c,d,e\}|=4$, then assume by symmetry that $4\in \{c,d,e\}$ and $5\not\in \{c,d,e\}$.
If $a=5$, then $b\neq 5$ and we can color $(u,uv)$ with 4 and $(v,uv)$ with 5 and obtain a conditional incidence $5$-coloring of $G$.
Hence $a\neq 5$.
If we are able to color $(u,uv)$ with $a'\in \{4,5\}\setminus\{a\}$ and $(v,uv)$ with $b'\in [5]\setminus\{b,c,d,e,a'\}$ such that $5\in \{a',b'\}$, then we obtain a conditional incidence $5$-coloring $\varphi$ of $G$ as $\varphi(I_u\cup A_u)=[5]$.
If this is impossible, then $a=4$ and $\{b,c,d,e\}=[4]$. It follows that $b,c\not\in \{4,5\}$.
Recolor $(u,ux)$ with 5 and color $(u,uv)$ with 1 and $(v,uv)$ with 5. This completes a conditional incidence $5$-coloring of $G$.





If $|\{1,2,3\}\cup\{c,d,e\}|=5$, then color $(u,uv)$ with $a'\in \{4,5\}\setminus\{a\}$ and $(v,uv)$ with $b'\in [5]\setminus\{b,c,d,e,a'\}$. 
This is possible as if $\{b,c,d,e,a'\}=[5]$ then $a'\not\in \{c,d,e\}$, which contradicts the assumption that $\{4,5\}\subseteq \{c,d,e\}$.
\end{proof}

\proof[The proof of Theorem \ref{thm:outerplanar}]
Suppose for a contradiction that 
$G$ is a minimal outerplanar graph in terms of $|V(G)|+|E(G)|$ with $\Delta(G)\leq \Delta$ 
such that $G$ is not conditional incidence $\Delta$-colorable.
It follows that $G$ is a $\varrho$-$\Delta$-critical graph.
As an outerplanar graph, $G$ contains one of the four configurations among \ref{c1}, \ref{c2}, \ref{c3}, and \ref{c4} by Lemma \ref{lem:outerplanar}.
However, $G$ does not contain \ref{c1} by Lemmas \ref{lem:relationship} and \ref{lem:minimumdegree2}, \ref{c2} by Lemmas \ref{lem:relationship} and \ref{lem:2-no-minus2}, \ref{c3} by Lemmas \ref{lem:relationship}, \ref{lem:2-no-minus2}, \ref{lem:2-no-twominus1} and \ref{lem:no234triangle}, or 
\ref{c4} by Lemmas \ref{lem:relationship}, \ref{lem:minimumdegree2}, \ref{lem:no234triangle}, \ref{no-T1}, \ref{no-T2} and \ref{no-T3}.
This contradiction completes the proof. \hfill $\square$

To close this section, we prove the following.

\begin{thm}\label{thm:outerplanar3}
$\chi^d(G)=\Delta(G)$ for every integer $d\geq 2$ if $G$ is an outerplanar graph with $\Delta(G)\geq 2$. 
\end{thm}

\begin{proof}
It is sufficient to prove $\chi^2(G)\leq 3$ for every subcubic outerplanar graph by \eqref{rela} and by Theorems \ref{thm:path} and \ref{thm:outerplanar2}.
Suppose for a contradiction that $G$ is minimal counterexample in terms of $|V(G)|+|E(G)|$ to this result. Clearly, $G$ is connected, so by Lemma \ref{lem:outerplanar}, $G$ contains a vertex $u$ of degree at most 2.

If $\deg_G(u)=1$, then assume $N_G(u)=\{v\}$. By the minimality of $G$, 
$G-u$ has a $2$-defective incidence $3$-coloring $\varphi'$. We extend $\varphi'$ to a $2$-defective incidence $3$-coloring of $G$ by coloring $(v,uv)$ with $a\in [3]\setminus \varphi'(I_v)$ and $(u,uv)$ with $b\in [3]\setminus (S\cup \{a\})$, where $S$ is the set of colors used at most twice among $A_v$ under $\varphi'$. This is possible as $|\varphi'(I_v)|=\deg_G(v)-1\leq 2$ and $|S|\leq 1$.

If $\deg_G(u)=2$, then assume $N_G(u)=\{v,w\}$. By the minimality of $G$, $G-uv$ has a $2$-defective incidence $3$-coloring $\varphi'$.
 Since $|\varphi'(I_v)|=\deg_G(v)-1\leq 2$, we are able to color $(v,uv)$ with $a\in [3]\setminus \varphi'(I_v)$.  
If $a\in \varphi'(A_v)$, then color $(u,uv)$ with $b\in [3]\setminus \{\varphi'(u,uw),a\}$ to complete a 
$2$-defective incidence $3$-coloring of $G$.
If $a\not \in \varphi'(A_v)$, then $|\varphi'(A_v)|=2$ and therefore we are still able to color $(u,uv)$ with $b\in [3]\setminus \{\varphi'(u,uw),a\}$ to compelte the desired coloring.
\end{proof}

Remark \ref{rem1} implies that the condition of $d\geq 2$ in Theorem \ref{thm:outerplanar3} is necessary.
Combining Theorems \ref{thm:path}, \ref{thm:outerplanar2}, and \ref{thm:outerplanar3} together,
we deduce the following.

\begin{cor}
\begin{align*}
  \defi(G)= 
  \begin{cases}
        1
        &\text{if } \Delta\neq 1,3,
        \\
        2
        &\text{if } \Delta=3,
        \\
        \infty
        &\text{if } \Delta=1.
        \end{cases}
\end{align*}
if $G$ is an outerplanar graph with maximum degree $\Delta$.
\end{cor}

\begin{rem}
Since the proofs in this section are all constructive, they yield a polynomial-time algorithm for outputting a $d$-defective incidence $\Delta$-coloring whenever we input an outplanar graph with maximum degree $\Delta$ and an integer $d$ such that $\Delta\geq 4$ and $d\geq 1$, or $\Delta\in \{2,3\}$ and $d\geq 2$.
\end{rem}

We leave an open problem to close this paper. 

\begin{pblm}\label{problem}
Does every bridgeless subcubic outerplanar graph has a 1-defective incidence 3-coloring?
\end{pblm}

It would be interesting to point out that not every bridgeless subcubic graph is 1-defectively incidence 3-colorable.
In the literature, a \textit{snark} is a simple, connected, bridgeless cubic graph with chromatic index equal to 4.
There are many well-known snarks including the Petersen graph, which is the smallest snark.
We claim that $\chi^1(S)\geq 4$ for every snark $S$.

Suppose for a contradiction that $S$ admits a 1-defective incidence 3-coloring $\varphi$.
For every edge $uv$, let $c(uv):=\varphi(u,uv)+\varphi(v,uv)$ (mod $3$).
For every two edges $ux$ and $uy$ incident with $u$, if $c(ux)=c(uy)$, then 
$\{\varphi(u,ux),\varphi(x,ux)\}=\{\varphi(u,uy),\varphi(y,uy)\}$, which implies $\varphi(u,ux)=\varphi(y,uy)$, $\varphi(u,uy)=\varphi(x,ux)$, and thus 
$\varphi(u,uz)=\varphi(z,uz)$ for the third neighbor $z$ of $u$, a contradiction.
Hence the images under $c$ of every two adjacent edges of $S$ are distinct and thus the mapping $c:~E(S)\longrightarrow \{0,1,2\}$ is a proper edge $3$-coloring of $S$, contradicting the fact that the chromatic index of $S$ is 4.

Note that this argument cannot return a negative answer to Problem \ref{problem}, as every subcubic outerplanar graph has chromatic index 3.



%

\bibliographystyle{abbrv}
\bibliography{references}

\newpage

\end{document}